\documentclass[a4paper,12pt,leqno]{article}
\usepackage[T1]{fontenc}
\usepackage{amsmath, amssymb}
\usepackage[utf8]{inputenc} 
\usepackage{amsthm}
\usepackage{enumitem}
\newtheorem{thm}{Theorem}
\usepackage{appendix}
\newtheorem{lem}{Lemma}

\newtheorem{defa}{Definition}
\newtheorem{prop}{Proposition}

\begin{document}

\title{Lagrangian submanifolds of standard multisymplectic manifolds}
\author{Gabriel Sevestre\thanks{gabriel.sevestre@univ-lorraine.fr}~  
 and Tilmann Wurzbacher\thanks{tilmann.wurzbacher@univ-lorraine.fr}\\
\phantom{Higgs} \\
Institut \'Elie Cartan Lorraine \\
Universit\'{e} de Lorraine et C.N.R.S. \\
F-57000 Metz, France \\}

\date{September 2, 2018}

\maketitle

\begin{abstract}
We give a detailed, self-contained proof of Geoffrey Martin's normal form theorem for Lagrangian 
submanifolds of standard multisymplectic manifolds (that generalises Alan Weinstein's famous normal
form theorem in symplectic geometry), providing also complete proofs for the necessary results in
foliated differential topo\-logy, i.e., a foliated tubular neighborhood theorem and a foliated relative Poincar\'e lemma.
\end{abstract}

\noindent {\bf MSC (2010)} Primary: 53D05, 53D12; Secondary: 53C12\\

\noindent {\bf Keywords:} multisymplectic geometry, Lagrangian submanifolds, foliated differential topology


\section*{Introduction}
\addcontentsline{toc}{section}{Introduction}

It is well-known that Lagrangian submanifolds play a central role in symplectic geometry. This can easily be traced back to the search 
for so-called "generating functions" of (local) symplectomorphisms in the framework of the Hamilton-Jacobi method for integrating 
Hamilton's equation (see the classical reference \cite{arno}, Sections 47-48). 
This method is closely connected to the observations 
that the graph of a diffeomorphism between two symplectic manifolds 
is Lagrangian if and only if the diffeomorphism is symplectic and that the image of a one-form is Lagrangian (inside the cotangent bundle) if and
only if the form is closed. Alan Weinstein deduced from such classical facts his famous symplectic creed: "Everything is a Lagrangian submanifold".
Of course, from a modern perspective the main argument for this creed is ... Weinstein's fundamental result from 1971 (see \cite{weinlag}):\\

\noindent {\bf Weinstein's normal form theorem.} {\it Let $L$ be a Lagrangian submanifold of a symplectic manifold $(M,\omega)$. 
Then there exist open neighborhoods 
$U$ and $V$ of $L$ in $M$ respectively $T^*L$, and a diffeomorphism $\phi: U\to V$ such that $\phi\vert_L=id_L$ and
$\phi^*(\omega^{T^*L})=\omega$ on $U$.}\\
  
Classical mechanics is geometrized by the Hamiltonian approach on cotangent bundles and more generally on symplectic manifolds, 
whereas its 
higher dimensional analogue, classical field theory, can be formulated in a Hamiltonian way on multicotangent or jet bundles, and leads more
generally to multisymplectic manifolds (cf. \cite{leotil}, Section 2 for a recent account of this). A multisymplectic manifold is a manifold together with
a nondegenerate, closed $(k{+}1)$-form $\omega$ with $k$ in $\mathbb N$; $k=1$ being the symplectic case.\\

In a 1988  article (\cite{Martin1988}) Geoffrey Martin extended Weinstein's result to an important class of multisymplectic manifolds 
including multicotangent bundles. (Note that he reserves the term ``multisymplectic" for the class of multisymplectic manifolds 
where his theorem 
applies.) The proof of his main result (Lemma 2.1) being rather cryptic, and in parts being reduced to mere hints for the reader, 
his precocious results fell into oblivion, not receiving the deserved attention.\\

The spanish school on differential-geometric methods in mathematical physics revived multisymplectic geometry 
(in its modern definition) at the end of
the last century, and Manuel de L\'eon, David Mart\'\i n de Diego and Aitor Santamar\'\i a-Merino gave in \cite{leodieg} a 
rather detailed framework for 
Martin's normal form theorem. Unfortunately, the necessary condition that a certain naturally associated subbundle of the 
tangent bundle of the ambient
manifold should be integrable is not emphasised in their proof of Martin's main result (see the proof of Lemma 3.24 
in the cited article).\\

Since multisymplectic geometry is by now emerging fast as the "right"  (higher) geometric formulation of classical field theory,
thanks to the advent of rather well-suited homotopical and homological methods, the interest in Martin's result is growing and 
we felt compelled to give a 
self-contained, detailed account of his result and techniques. It turns out that one crucially needs 
"folkloristic" extensions of two 
standard theorems in differential topology to a foliated setting (these being of independent interest, in fact). Once established, 
Martin's ingenious idea that 
the path method of J\"urgen Moser (see \cite{mosvol}) applies though a multisymplectic form of degree $k{+}1$ does not yield an isomorphism between
the tangent bundle and the bundle of $k$-forms, goes through and yields the following result:\\

\noindent {\bf Martin's normal form theorem (Theorem 1 below).} {\it Let $(M,\omega)$ be a standard $k$-plectic manifold, with $k>1$. Let the
distribution $W\subset TM$ be defined as  $W:=\cup_{p\in M} W_{\omega}(p)$ and let $L$ be a $k$-Lagrangian 
submanifold of $M$ complementary to $W$ (that is $T_pM=T_pL\oplus W|_p$, $\forall p\in L$).
If $W$ is integrable, there exist open neighborhoods $U$ and $V$ of $L$ in $M$ 
and $\Lambda^k(T^{*}L)$, and a diffeomorphism $\phi:U\rightarrow V$ such that:
\begin{equation*}
\phi|_{L}=id_{L}   \,\, \mbox{and} \,\,  \phi^{*}(\omega^{\Lambda^k(T^{*}L)})=\omega \,\,  \mbox{on} \,\, U.
\end{equation*}  }

Results of a related but more global nature were obtained by Frans Cantrijn, Alberto Ibort and Manuel de L\'eon in 1999
(see Theorem 7.3 in \cite{cil1999}) and Michael Forger and Sandra Z. Yepes in 2013 (see Theorem 7 in  \cite{Forger:2012pr}). In both 
cases the focus is shifted from the local situation near a Lagrangian submanifold to the foliation associated to an involutive Lagrangian 
distribution and its leaf space, implying an important role for regularity assumptions on the foliation, and for connections on the leaves.\\

We conclude the introduction by summarising the paper's content. In Section 1 we give the basic definitions, as multisymplectic vector 
spaces and manifolds and their isotropic and Lagrangian subspaces respectively submanifolds. We also give here some examples of 
isotropic and Lagrangian submanifolds of multisymplectic manifolds. Section 2 introduces the notions of "standard" multisymplectic 
vector spaces and manifolds, central for this article. We prove the fundamental properties of a standard multisymplectic vector space
$(V,\omega)$ (with $\omega$ a $(k{+}1)$-linear form and $k>1$), notably the existence of a unique subspace $W\subset V$ that is 
isomorphic to  $\Lambda^k (V/W)^*$ via the natural contraction map (compare Lemma 1 and Proposition 1). On the level of manifolds, we 
explain why multicotangent bundles are standard multisymplectic manifolds. In Section 3 we give a detailed proof of Martin's normal form theorem 
(see above), expanding and explaining Martin's extremely brief original proof. In an Appendix we give complete proofs for the extension of 
two classical differential-topological results to foliated manifolds, more precisely, we show a foliated tubular neighborhood theorem and a foliated relative 
Poincar\'e lemma (see Theorems 2 and 3).\\

\noindent {\it Acknowledgements.}  We wish to thank Camille Laurent-Gengoux for several useful discussions related to the content of this article.


\section{Multisymplectic vector spaces and manifolds,  and Lagrangian submanifolds}

In this section, we give the basic definitions used in the paper, together with some examples. We will work over the real numbers and 
all manifolds will be smooth. 
The algebraic considerations for vector spaces hold true over fields of characteristic zero instead of the reals.

\begin{defa}
Let $V$ be a vector space, $k\geq 1$ and $\omega\in\Lambda^{k+1}(V^*)$. We say that
$(V,\omega)$ is a k-plectic vector space 
(or simply a multisymplectic vector space) if $\omega$ is nondegenerate, in the sense that :
\begin{equation*}
\omega^{\sharp}:V\rightarrow \Lambda^k(V^*); \ v\mapsto \iota_v\omega
\end{equation*}
is injective. 
\end{defa}

As in the symplectic case, we can define orthogonal subspaces with respect to $\omega$, but in this setting
we have more than just one "$\omega$-orthogonal complement" for a given subspace of $V$:

\begin{defa}

Let $(V,\omega)$ be a $k$-plectic vector space, $U\subset V$ a subspace and $1\leq j\leq k$.
We define the $j$-th orthogonal complement of $U$ with respect to $\omega$ as follows:

\begin{equation*}
U^{\perp,j}:=\{v\in V \ | \ \iota_{v\wedge u_1\wedge...\wedge u_j}\omega=0, \ \forall u_1,...,u_j\in U \}.
\end{equation*}
We say that $U\subset V$ is a $j$-isotropic subspace (respectively, a $j$-Lagrangian subspace) if 
$U\subset U^{\perp,j}$ (respectively if $U=U^{\perp,j}$). 

\end{defa}

Going to manifolds we have:

\begin{defa}
Let $M$ be a manifold and $\omega\in \Lambda^{k+1}T^*M$. We say that $(M,\omega)$ is a $k$-plectic manifold, 
or simply a multisymplectic manifold, if the form $\omega$ is closed and nondegenerate, in the sense that for all $q\in M$, the map :

\begin{equation*}
\omega_q^{\sharp}:T_qM\rightarrow \Lambda^k(T_q^*M); \ v_q\mapsto \iota_{v_q}\omega_q
\end{equation*}
is injective. 

\end{defa}

Analogously to the linear case, we will say that a regular submanifold $L$ is a $j$-isotropic respectively $j$-Lagrangian 
submanifold of $M$, if, for each $p\in L$, $T_pL$ is a $j$-isotropic respectively $j$-Lagrangian subspace 
of $T_pM$. \newline

Before studying a special class of multisymplectic manifolds in Sections 2 and 3, we will give general examples of 
multisymplectic manifolds and isotropic submanifolds. Note that if $N$ is a submanifold of $M$ of dimension $n$, then 
$N$ is $j$-isotropic for all $j\geq n$ in a trivial way. 
Thus in the following examples, we will only consider "interesting" isotropic and Lagrangian submanifolds, where this is not the case.\newline

\noindent \textbf{Example 1}. Let $M$ be an orientable manifold of dimension $m$ and $\omega$ a volume form on $M$. Then $(M,\omega)$ 
is a $(m{-}1)$-plectic manifold. In this case there are no non-trivial examples (in the sense stated above) of isotropic submanifolds of $M$. 
 \newline

\noindent \textbf{Example 2}. Let $Q$ be a manifold, $k \geq 1$ and the dimension of $Q$ being greater or equal to $k{+}1$. 
Then the manifold $M:=\Lambda^k(T^*Q)$ is naturally equipped with a $k$-plectic form. 
Indeed let $\theta\in \Omega^k(M)$ be defined by :
\begin{equation*}
\theta_{\alpha_p}(v_1,...,v_k):=\alpha_p(\pi_{*\alpha_p}(v_1),...,\pi_{*\alpha_p}(v_k)),
\end{equation*}
where $\alpha_p\in M$, $v_j\in T_{\alpha_p}(M)$, and $\pi:M\rightarrow Q$ is the canonical projection. Then $\omega:=-d\theta$ 
is a $k$-plectic form on $M$. This construction is the generalization of the symplectic form on a cotangent bundle. The 
zero-section of $\Lambda^k(T^*M)$ is a $k$-Lagrangian manifold, and the fibers of $\pi$ are $1$-Lagrangian. To see this, we can work 
in local coordinates. A direct computation shows then that if $(q^i)$ are coordinates on an open subset $U\subset Q$ and 
$(p_I)$ are coordinates on the fibers of $\Lambda^k(T^*U)$, we have :
\begin{equation*}
\omega|_U=- \sum_{i_1,...,i_k} dp_{i_1,...,i_k}\wedge dq^{i_1}\wedge...\wedge dq^{i_k}.
\end{equation*}
Using  this local description it is easy to see that $Q$ is $k$-Lagrangian and the fibers are $1$-Lagrangian. More generally, 
for $\alpha\in \Omega^k(Q)$, we have that $im(\alpha)\subset M$ is a $k$-Lagrangian manifold if and only if $\alpha$ is closed. 
This follows from $\alpha^*\theta=\alpha$ (where, on the left-side, $\alpha$ is regarded as a map $\alpha:Q\rightarrow M$), 
implying $\alpha^*\omega=- d\alpha$. \newline

\noindent \textbf{Example 3}. Let $(M,\eta)$ be a $k$-plectic manifold and $\omega\in\Omega^{k{+}1}(M\times M)$ the form given by:

\begin{equation*}
\omega=p_1^*\eta - p_2^* \eta,
\end{equation*}

\noindent where for $i=1,2$, the map $p_i$ is the projection $p_i: M\times M\rightarrow M$ on the $i$-th factor. Then $(M\times M,\omega)$ is a 
$k$-plectic manifold. Considering a diffeomorphism $\phi:M\rightarrow M$, we claim that $\Gamma_{\phi}$, the graph of $\phi$,
 is $k$-Lagrangian if and only if $\phi$ is a symplectomorphism in the sense that $\phi^*\eta=\eta$. Indeed $T_{(q,\phi(q))}(M\times M)=\{(u_q,\phi_{*q}(u_q)) \ | \ u_q\in T_q M\}$. Then for $(u_i,\phi_*(u_i))\in T_{(q,\phi(q))}(M\times M)$ ($1\leq i \leq k$) we obtain:
\begin{align*}
\omega_{(q,\phi(q))}((u_1,\phi_*(u_1)),...,(u_k,\phi_*(u_k))&=\eta_q(u_1,...,u_k) - \eta_{\phi(q)}(\phi_*(u_1),...,\phi_*(u_k)) \\
&=\eta_q(u_1,...,u_k)-(\phi^*\eta)_q(u_1,...,u_k),
\end{align*}
showing the claim. \newline

\noindent \textbf{Example 4}. Let $M$ be a complex manifold with a holomorphic volume form 
$\Omega$. Then setting $\omega=\Re(\Omega)$, the real part of $\Omega$, 
turns $(M,\omega)$ into a multisymplectic manifold. To get a feeling of how Lagrangian submanifolds may look in this case,
we consider $M=\mathbb{C}^3=\mathbb{R}^6$ and $\Omega=dz^1\wedge dz^2 \wedge dz^3 = dz^{123}$. We find:
\begin{equation*}
\omega=dx^{123}-dx^{156}-dx^{246}-dx^{345},
\end{equation*}
\newline
where we have omitted wedge products and $x^i$ are coordinates in $\mathbb{R}^6$. Then the manifold $\{x^1=x^2=x^3=0\}$ 
is 2-Lagrangian, and the manifold $\{x^2=x^3=x^5=x^6=0\}$ is 1-Lagrangian. \newline

\noindent \textbf{Example 5}. (Compare \cite{cil1999}, Section 3.) 
Let $M=\mathbb{R}^6$ and 
\begin{equation*}
\omega=dx^{145}+dx^{246}+dx^{356}+dx^{456}.
\end{equation*}
\newline
Then $\omega$ is a $2$-plectic form, and $L_2=\{ x^1=x^3=x^4=x^6=0 \}$ and 
$L_3=\{ x^4=x^5=x^6= 0\}$ are (linear)
$1$-Lagrangian submanifolds of $(M,\omega)$ of different dimensions. (Note that $(\mathbb{R}^6, \omega)$
is symplectomorphic to the multicotangent bundle $\Lambda^2(T^*\mathbb{R}^3)$ with the multisymplectic 
form defined in Example 2 above.)
\newline

\noindent \textbf{Example 6}. Let $G$ be a real semi-simple, compact Lie group. Consider the Cartan form $\omega\in\Omega^3(G)$, 
which is the bi-invariant form defined at the neutral element $e$ by :
\begin{equation*}
\omega_e(\xi,\eta,\zeta):=\left<[\xi,\eta],\zeta\right>,
\end{equation*}
where $\xi,\eta,\zeta\in \mathfrak{g}$ (the Lie algebra of $G$) and $<.,.>$ is the Killing form. The form $\omega$ is closed because it is 
bi-invariant, and it is nondegenerate because the Killing form is nondegenerate and $[\mathfrak{g},\mathfrak{g}]=\mathfrak{g}$. 
Consider $T\subset G$, a torus. Its Lie algebra $\mathfrak{t}$ is abelian and thus $T$ is 1-isotropic. Thus if $T$ is a maximal torus then 
it is 1-Lagrangian.


\section{Standard multisymplectic vector spaces and manifolds}

In this section, we will be interested in a special class of multisymplectic vector spaces and manifolds, important in applications 
of multisymplectic geometry to classical field theories.

\begin{defa}
Let $V$ be a vector space and $k>1$. We say that $V$ is a standard $k$-plectic vector space, if $(V,\omega)$ is a $k$-plectic vector space and there exists a subspace $W\subset V$ such that:
\begin{enumerate}[label=(\arabic*)]
\item $\forall u,v\in W, \ \iota_{u\wedge v}\omega=0$
\item $dim(W)=dim(\Lambda^k((V/W)^*))$
\item $codim(W)>k\,$.
\end{enumerate}
\end{defa}

Let us also consider the following condition :
\begin{equation*}
 dim(W)\geq codim(W)\, .  \tag{\textit{3'} }
\end{equation*}

Concentrating on the higher degree cases ($k>1$) we then have the following relations between these conditions:

\begin{lem}
Let $(V,\omega)$ be a $k$-plectic vector space with $k>1$. Then :
\begin{enumerate}[label=(\roman*)]
\item conditions $(1)$ and $(2)$ imply that the map $\omega^{\sharp}|_{W}:W\rightarrow \Lambda^kV^*$ induces a linear isomorphism :
\begin{equation*}
\chi:W\rightarrow \Lambda^k(V/W)^*, 
\end{equation*}
\item if conditions $(1)$ and $(2)$ are satisfied, then condition $(3)$ is equivalent to condition $(3')$,
\item if $(V,\omega)$ is standard, then $dim(W)\geq 2$.
\end{enumerate}
\end{lem}

\noindent {\bf Remark.} In reference \cite{leodieg} these multisymplectic vector spaces are called of type $(k{+}1,0)$.

\begin{proof}
We denote the projection $V\rightarrow V/W$ by $\pi$. Then the subspace $\pi^*(\Lambda^k(V/W)^*)\subset \Lambda^kV^*$ is given by :

\begin{equation*}
\pi^*(\Lambda^k(V/W)^*)=\{\eta\in \Lambda^kV^* \ | \ \iota_v\eta=0, \ \forall v\in W\}.
\end{equation*}
By condition $(1)$, $\omega^{\sharp}(w)\in \pi^*(\Lambda^k(V/W)^*)$ whenever $w\in W$ ; thus $\omega^{\sharp}$ induces a 
unique injective linear map $\chi:W\rightarrow \Lambda^k(V/W)^*$ such that $\pi^*\circ\chi=\omega^{\sharp}|_W$. Moreover, $\chi$ is a linear isomorphism by condition $(2)$ ; thus proving the first assertion. 

Now put $d=dim(W)$, $c=codim(W)$. Then $dim(V)=c+d$. Assume conditions $(1)-(3)$ to hold, and $d<c$. By condition $(2)$, 
$d=\binom{c}{k}$, thus $d=\frac{c(c-1)...(c-(k-1))}{k(k-1)...1}\geq c$. This contradiction shows that conditions $(1)-(3)$ imply condition $(3')$. 

Now assume conditions $(1),(2),(3')$ to hold and $c\leq k$. Since $c<k$ is easily seen to contradict $(2)$, then $c=k$ and $d=1$. By $(3')$, $dim(V)=2$ and thus $1=dim(\Lambda^k(V/W)^*)$, implying $k=1$, contradicting the assumptions.
Thus the conditions $ (1),(2),(3')$ imply conditions $(1)-(3)$, and the second 
assertion is proven. 

Now assume that $d\leq1$. If $d=0$, then $\Lambda^kV^*=\{0\}$, contradicting the fact that $\omega$ is nondegenerate. Now if $d=1$, we have $1=\binom{dim(V)-1}{k}$ and therefore $k+1=dim(V)$, implying $c=k$ ; this contradiction proves the last assertion.
 
\end{proof}

If $(V,\omega)$ is a standard $k$-plectic vector space, with $k>1$, then the subspace $W$ satisfying Definition 2 is unique as shows the following :

\begin{prop}
Let $(V,\omega)$ be a standard $k$-plectic vector space, with $k>1$, and $W,\widetilde{W}$ two subspaces satisfying Definition 2. Then $W=\widetilde{W}$.
\end{prop}

\begin{proof}
First we show that $W\cap \widetilde{W}$ has codimension at most $1$ in $\widetilde{W}$. To do this, assume the opposite : $codim_{\widetilde{W}}(W\cap\widetilde{W})>1$. Then, there exists linearly independent vectors $u,v$ of $\widetilde{W}$ such that $span(u,v)\cap W=\{0\}$ ; thus we can find $\eta\in \Lambda^k(V/W)^*$ such that $\iota_{u\wedge v}\eta\ne0$. But, for all $w\in W$, $\iota_{w}\iota_{u\wedge v}\omega=0$, so there cannot exist a $w\in W$ such that $\eta=\iota_w\omega$, and this contradicts the fact that the map $\chi$ is an isomorphism. 

Now suppose $W\ne \widetilde{W}$. Then there exists a non-zero vector $z\in \widetilde{W}$ such that $span(z)\cap W=\{0\}$. For all $w\in W\cap\widetilde{W}$, $\eta\in \Lambda^{k-1} (V/W)$ :

\begin{equation*}
\chi^*(\pi(z)\wedge \eta)(w)=(\pi(z)\wedge\eta)(\chi(w))=\omega(w,z,\eta)=0,
\end{equation*}
where $\chi^*$ denotes the dual of the map $\chi$, and $\pi:V\rightarrow V/W$ is the canonical projection. The above equation is well-defined because for $w\in W$, $\iota_w\omega$ depends only on its evaluation on element of $\Lambda^{k} (V/W)$, because of condition $(1)$ in Definition 2. Denote $Z=span(z)$. The above computation shows that:

\begin{equation*}
\chi^*(\pi(Z)\wedge\Lambda^{k-1}(V/W))\subset ann_{W^*}(W\cap\widetilde{W}),
\end{equation*}
where $ann_{W^*}(W\cap\widetilde{W})=\{\eta\in W^* \ | \ \eta(w)=0, \ \forall w\in W\cap\widetilde{W}\}$. This implies, together with $ codim_{\widetilde{W}}(W\cap\widetilde{W})\leq1$, that $dim(\chi^*(\pi(Z)\wedge\Lambda^{k-1}(V/W)))\leq 1$. Furthermore :

\begin{equation*}
dim(\pi(Z)\wedge\Lambda^{k-1}(V/W))=dim(\Lambda^{k-1}(V/W))>1,
\end{equation*}
because $codim(W)>k$. This shows a contradiction, and thus the Proposition. 

\end{proof}

The preceding proposition allows to denote such a subspace by $W_{\omega}$ and motivates the next definition :

\begin{defa}
Let $M$ be a manifold, $k>1$,  and $\omega\in\Lambda^{k+1}(T^*M)$. We say that $(M,\omega)$ is a standard $k$-plectic manifold 
if $(M,\omega)$ is a $k$-plectic manifold and if for each $p\in M$, $(T_pM,\omega_p)$ is a standard $k$-plectic vector space. For all 
$p\in M$ the unique subspace   of $T_pM$ satisfying Definition 4 is denoted by $W_{\omega}(p)$ or simply $W(p)$.
\end{defa}

The remainder of this section is dedicated to showing that standard multisymplectic vector spaces are in fact symplectomorphic to a 
canonical $k$-plectic model that we will describe now.

\begin{prop}
Let $(V,\omega)$ be a standard $k$-plectic vector space. Then the subspace $W_{\omega}$ is $1$-Lagrangian. Moreover, there 
exists a $k$-Lagrangian vector space $L\subset V$ complementary to $W_{\omega}$ and the map $\chi$ induces (for all choices 
of such $L$) an isomorphism:
\begin{equation*}
W_{\omega}\cong \Lambda^k(L^*). 
\end{equation*}
\end{prop}

\begin{proof}

Condition $(1)$ in Definition 2 implies that $W_{\omega}$ is $1$-isotropic. Now if $w\in W_{\omega}^{\perp,1}$ but 
$w\notin W_{\omega}$, then we can find $\eta\in \Lambda^k(V/W_{\omega})^*$ such that $\iota_w\eta\ne0$. 
But, for all $u\in W_{\omega}$, $\iota_u\iota_w\omega=0$ and thus there cannot exist a $u\in W_{\omega}$ such 
that $\eta=\iota_u\omega$. This property contradicts the fact that the map $\chi$ is an isomorphism, and therefore
proves that $W_{\omega}$ is $1$-Lagrangian.

Now let $\widetilde{L}$ be any subspace complementary to $W_{\omega}$. We may canonically identify $V/W_{\omega}$ and 
$\widetilde{L}$ since the restriction to $\widetilde{L}$ of the projection $\pi:V\rightarrow V/W_{\omega}$ is an isomorphism. Thus 
we have a canonical isomorphism $\chi:W_{\omega}\stackrel{\cong}{\longrightarrow} \Lambda^k(\widetilde{L}^*)$. We will search 
for a $k$-Lagrangian complement of the form $L=\{v+Av \ | \ v\in \widetilde{L}\}$ for some linear map 
$A:\widetilde{L}\rightarrow W_{\omega}$. For $L$ to be $k$-Lagrangian, it has to verify that $L\subset L^{\perp,k}$, i.e.,
for all $v_j\in \widetilde{L}$, $j=1,...,k{+}1$ :
\begin{equation*}
\omega(v_1+A v_1,...,v_{k+1}+A v_{k+1})=0.
\end{equation*}
This condition suffices here because assuming that there exists an element $u\in L^{\perp,k} \backslash L$, we may write 
$u=v+w$, with $v\in L$ and $w\in W_{\omega}$. For $u_1,...,u_k\in L$ we compute then:
\begin{align*}
\omega(u,u_1,...,u_k)&=\omega(v,u_1,...,u_k)+\omega(w,u_1,...,u_k) \\
&=\omega(w,u_1,...,u_k)=0
\end{align*}
because $v\in  L\subset L^{\perp,k}$, and $u\in L^{\perp,k}$. Thus we obtain $w\in W_{\omega}\cap L^{\perp,k}$. But then for 
all $v_j=x_j+y_j$, $v_j\in V$, $x_j\in L$ and $y_j\in W_{\omega}$ :
\begin{align*}
\omega(w,v_1,...,v_k)&=\omega(v,u_1,...,u_k)+\omega(w,u_1,...,u_k) \\
&=0.
\end{align*}
By the nondegeneracy of $\omega$, we obtain $w=0$, thus $u=v\in L$ which is a contradiction to $u\in L^{\perp,k} \backslash L$.\\ 

\noindent Now we return to the construction of the linear map $A$. We have:
\begin{align*}
\omega(v_1+A v_1,...,v_{k+1}+A v_{k+1})&=\omega(v_1,...,v_{k+1}) \\
&+ \sum_{j=1,...,k+1} (-1)^{j+1}\omega(Av_j,v_1,...,\widehat{v_j},...,v_{k+1}).
\end{align*}
Let $\Phi:=\chi\circ A$, then the Lagrangian condition reads as follows:
\begin{equation*}
\omega(v_1,...,v_{k+1})=- \sum_{j=1,...,k+1} (-1)^{j+1}\Phi(v_j)(v_1,...,\widehat{v_j},...,v_{k+1}).
\end{equation*}
We denote by $T$ the application $T:\widetilde{L}\rightarrow \Lambda^k(\widetilde{L}^*), \ v\mapsto \iota_v\omega$. 
If $\Phi=-\frac{1}{k+1}T$, then the above condition is verified. Thus the map $A:=(\chi)^{-1}\circ \Phi$ 
has the property that its graph $L=\Gamma_{\Phi}$ is a $k$-Lagrangian space, complementary to $W_{\omega}$.
\end{proof}

\begin{defa}
Let $V$ be a vector space and $\omega_{can}$ the canonical $(k{+}1)$-form on the space $\mathcal{V}:=V\oplus \Lambda^k(V^*)$
given by:
\begin{equation*}
\omega_{can}(v_1\oplus \alpha_1,...,v_{k+1}\oplus\alpha_{k+1})=
\sum_{i=1}^k (-1)^{i+1}\alpha_{i}(v_1,...,v_{i{-}1},\widehat{v_i},v_{i+1},...,v_{k+1}),
\end{equation*}
for all $v_j\in V$, and $\alpha_j\in \Lambda^k(V^*)$. Then $\omega_{can}$ is a $k$-plectic form. We call $(\mathcal{V},\omega_{can})$ 
a canonical 
$k$-plectic vector space. 
\end{defa}

\begin{lem}
Let $(V,\omega)$ be a $k$-plectic vector space with $k>1$. Then $(V,\omega)$ is isomorphic to a canonical  $k$-plectic vector space 
$(L\oplus\Lambda^k(L^*),\omega_{can})$ if and only if $(V,\omega)$ is standard.
\end{lem}

\begin{proof}
Let $(V,\omega)$ be a standard $k$-plectic vector space, and $L$ a $k$-Lagrangian subspace complementary to $W_{\omega}$. 
We define $\omega_{can}$ as above on the space $L\oplus \Lambda^k(L^*)$. Let 
$\gamma:=id_L\oplus \chi:L\oplus W_{\omega}\rightarrow L\oplus\Lambda^k(L^{*})$, where we again canonically 
identified  $V/W_{\omega}$
and $L$. Then $\gamma$ is a linear isomorphism 
and furthermore we find:
\begin{align*}
\omega_{can}(\gamma(u_1\oplus w_1),...,\gamma(u_{k+1}\oplus w_{k+1}))&=\omega_{can}(u_1\oplus \iota_{w_1}\omega),...,u_{k+1}\oplus \iota_{w_{k+1}}\omega)) \\
&= \sum_{j=1...k+1}(-1)^{j+1}\iota_{w_j}\omega(u_1,...,\widehat{u_j},...,u_{k+1}) \\ 
&=\omega(u_1\oplus w_1,...,u_{k+1}\oplus w_{k+1}),
\end{align*}
i.e. $\gamma^*\omega_{can}=\omega$. Therefore, the $k$-plectic space $(V,\omega)$ is symplectomorphic to 
$(L\oplus \Lambda^k(L^*),\omega_{can})$. \\

\noindent To show the converse, first note that if $(L\oplus\Lambda^k(L^*),\omega_{can})$ is 
a canonical $k$-plectic vector space, then the space $L$ identified with $L\times\{0\}\subset L\oplus \Lambda^k(L^*)$, is $k$-Lagrangian, 
and the space $ \Lambda^k(L^*)$, identified with $\{0\}\times \Lambda^k(L^*)\subset L\oplus \Lambda^k(L^*)$ is $1$-Lagrangian. Indeed :

\begin{equation*}
\omega_{can}((v_1,0),...,(v_{k+1},0))=0,
\end{equation*}
and if $\omega_{can}((v,\alpha),(v_1,0),...,(v_k,0))=0$ for all $v_j\in L$, then $\alpha(v_1,...,v_k)=0$ and then $\alpha=0$. 
This shows that $L$ is $k$-Lagrangian. Moreover :

\begin{equation*}
\omega_{can}((0,\alpha),(0,\beta),(v_1,\gamma_1),...,(v_{k-2},\gamma_{k-2})=0,
\end{equation*}
and if $\omega_{can}((v,\alpha),(0,\beta),(v_1,\gamma_1),...,(v_{k-2},\gamma_{k-2})=0$ for all $v_j\in V$ and $\gamma_j\in\Lambda^k(L^*)$, then $\iota_v\beta=0$ and thus $v=0$. This shows that $\Lambda^k(L^*)$ is $1$-Lagrangian. Thus, if a $k$-plectic linear space $(V,\omega)$ is symplectomorphic to a space $(L\oplus\Lambda^k(L^*),\omega_{can})$, then, pulling back the $1$-Lagrangian space $\Lambda^k(L^*)$ to $V$ with this symplectomorphism gives a linear subspace $W\subset V$ satisfying Definition 2.
\end{proof}

\noindent {\bf Remark.} Consider now a manifold $Q$, $k>1$ and $(\Lambda^k(T^*Q),\omega^{\Lambda^k(T^*Q)})$ the $k$-plectic 
structure exposed in Example 2 of Section 1. We have :
\begin{equation*}
T(\Lambda^k(T^*Q))|_Q=TQ\oplus \Lambda^k(T^*Q).
\end{equation*}
Recall that each fiber of $\Lambda^k(T^*Q)$ is $1$-Lagrangian, and $Q$ is $k$-Lagrangian. Thus, at each point $p\in Q$, the form 
$\omega^{\Lambda^k(T^*Q)}$ evaluated at the point $p$ is in fact the canonical form on the space $T_pQ\oplus \Lambda^k(T_p^*Q)$. Using the coordinate expression of the form :
\begin{equation*}
\omega^{\Lambda^k(T^*Q)}|_U=- \sum_{i_1,...,i_k} dp_{i_1,...,i_k}\wedge dq^{i_1}\wedge...\wedge dq^{i_k},
\end{equation*}
with $U\subset Q$ an open set, $(q^i)$ coordinates on $U$ and $(p_i)$ coordinates on the fibers of $T^*U$, we see that 
at any point $\alpha_q$ of 
$\Lambda^k(T^*Q)$, the fiber $\Lambda^k(T_q^*Q)$ is $1$-Lagrangian and the tangent space at $\alpha_q$
satisfies the conditions of Definition 4. This implies that 
$(\Lambda^k(T^*Q),\omega^{\Lambda^k(T^*Q)})$ is a standard $k$-plectic manifold.


\section{Normal forms of Lagrangian submanifolds of standard multisymplectic manifolds}

In this section we give a proof of the main result, that first appeared in equivalent form  as Lemma 2.1 
in Geoffrey Martin's 1988 article \cite{Martin1988}:

\begin{thm}
Let $(M,\omega)$ be a standard $k$-plectic manifold with $k>1$. Let the
distribution $W\subset TM$ be defined as  $W:=\cup_{p\in M} W_{\omega}(p)$ and let $L$ be a $k$-Lagrangian 
submanifold of $M$ complementary to $W$ (that is $T_pM=T_pL\oplus W_p$ for all $p$ in $L$).
If $W$ is integrable, there exist open neighborhoods $U$ and $V$ of $L$ in $M$ 
and $\Lambda^k(T^{*}L)$, and a diffeomorphism $\phi:U\rightarrow V$ such that:
\begin{equation*}
\phi|_{L}=id_{L}   \,\, \mbox{and} \,\,  \phi^{*}(\omega^{\Lambda^k(T^{*}L)})=\omega \,\, \mbox{on} \, \, U.
\end{equation*}
\end{thm}

\begin{proof}
It follows from the definition of a standard $k$-plectic manifold and the results found in the 
linear case, that if $L$ is complementary to $W$ we have an isomorphism of vector bundles :
\begin{equation*}
\chi : W|_{L} \rightarrow \Lambda^k(T^{*}L),
\end{equation*}
which is given by contraction of vectors in $W$ with $\omega$, using the identification $TM|_L/W|_L \cong TL$. (Note that we write in the
sequel often $W_x$ for the fiber of the vector bundle $W\to M$ over $x$ in $M$.) 

Using this map we construct the following vector bundle isomorphism :
\begin{equation*}
\Psi:TM|_L=TL\oplus W|_L \cong TL\oplus \Lambda^k(T^{*}L)=T(\Lambda^k(T^{*}L))|_L,
\end{equation*}
acting as the identity on vectors of $TL$, and transforming vectors of $W$ via $\chi$ (i.e., $\Psi=id_{TL}\oplus\chi$). Furthermore, this map is for each $x\in L$ an isomorphism between the multisymplectic vector spaces $(T_xM,\omega_x)$ and $(T_x(\Lambda^kT^*L),\omega^{\Lambda^kT^*L}_x)$. We now wish to find a diffeomorphism $f:U_1\rightarrow U_2$, where $U_1, U_2$ are neighborhoods of $L$ respectively in $M$ and in $\Lambda^k(T^{*}L)$, such that $f|_L=id_L$ and for every $x\in L$, $T_xf=\Psi_x$. 
Such a map $f$ then fulfills $(f^*\omega^{\Lambda^k(T^*L)})|_L=\omega|_L$. 

By the foliated tubular neighborhood theorem, we may find a neighborhood $U$ of $L$ in $M$, a neighborhood $V$ of the zero-section in $W|_L$, and a diffeomorphism $\phi$, which is the identity along $L$, maps each leaf of the foliation to a fiber of $W|_L\rightarrow L$, and has as its differential at any point of $L$ the identity. Let $f:=\chi \circ\phi$. Then $f$ maps $L$ to the zero section of $\Lambda^k(T^*L)$, and is a diffeomorphism onto an open subset $U'\subset \Lambda^k(T^*L)$ which contains $L$ (as the zero section). Furthermore we have for $x\in L$, $T_xf|_{W_x}=\chi$ and $T_xf|_{T_xL}=id_{T_xL}$ (where we have identified $T_x(W_x)$ with $W_x$). 
Thus we obtain for $x\in L$, $T_xf=\Psi_x$, and upon putting
$ \widetilde{\omega}:=f^*\omega^{\Lambda^k(T^*L)}$,we arrive, by the above said, at $\widetilde{\omega}|_L=\omega|_L$. 

We now want to show that for any vector fields $X,Y$ (defined in $U$) and tangent to $W$, $\iota_{X\wedge Y}\widetilde{\omega}=0$. Let $p\in U$ and $F_p$ be the leaf of the foliation defined by $W,$ which passes through $p$. This leaf also passes through a point of $L$, say $x$. Then $\phi$ maps this leaf to the space $W_x$, and thus $f$ maps $F_p$ to $\Lambda^k(T_x^*L)$. Thus if, $X_p,Y_p$ are vectors in $W_p$, we may consider that $T_pf(X_p),T_pf(Y_p)$ are vectors of $\Lambda^k(T_x^*L)$. Since this space is $1$-Lagrangian with respect to $\omega^{\Lambda^k(T^*L)}$, we find $\iota_{X\wedge Y}\widetilde{\omega}=0$.

Working on an open neighborhood $U$ of $L$ in $M$, we adapt now the well-known Moser path method (see \cite{mosvol}) to our situation. Let $\omega_t=\omega +t(\widetilde{\omega}-\omega)$ for $t\in [0,1]$. Then we have $\omega_t|_L=\omega|_L$ and $\frac{\partial}{\partial t}\omega_t=\widetilde{\omega}-\omega=:\omega'$ so $\omega'|_L=0$. Thus for $x\in L$, $\omega_t(x)$ is nondegenerate for all $t\in [0,1]$ and the set of points $(t,x)$ such that $\omega_t(x)$ is nondegenerate is an open subset of $\mathbb{R}\times M$. So, shrinking $U$ if necessary, we may suppose that $\omega_t$ is nondegenerate in $U$ for all $t\in [0,1]$. We also have that $d\omega_t=d\omega'=0$. By the relative Poincar\'e lemma, there exist a neighborhood $U$ of $L$ in $M$ and $\mu\in\Omega^k(U)$ with $d\mu=\omega'$ and $\mu|_L=0$. Moreover 
-upon using Theorem 3- we can choose $\mu$ such that $\iota_v\mu=0$ whenever $v\in W$, because $\omega'$ vanishes when contracted with two vectors of $W$ (because both $\omega$ and $\widetilde{\omega}$ have this property). Therefore $\mu$ may be interpreted as a section $U\rightarrow \Lambda^k(TM/W)^*$. Let us now take a look at the map :
\begin{equation*}
\omega_t^{\sharp}:W\rightarrow \Lambda^k(T^*M),
\end{equation*}
given by contraction of $\omega_t$ with vectors of $W$. For $u,v\in W$, lying over the same point, $\iota_{u\wedge v}\omega_t=0$ ; so $\omega_t^{\sharp}$ may be seen as a map :
\begin{equation*}
\omega_t^{\sharp}:W\rightarrow \Lambda^k(TM/W)^* \ ; \ u\mapsto \iota_u\omega_t,
\end{equation*}
and, by the nondegeneration of $\omega_t$, this map is injective, and thus is an isomorphism for dimensional reasons. Then for each $t\in [0,1]$, there exists a unique vector field $X_t$ (which take values in $W$) such that :
\begin{equation*}
\iota_{X_t}\omega_t +\mu=0
\end{equation*}
The association $(t,x)\rightarrow X_t(x)$ thus gives a (time-dependent) vector field tangent to $W$. But $\mu|_L=0$ so we deduce that for $x\in L$, for all $t\in [0,1]$, $X_t(x)=0$ by the nondegeneration of $\omega_t$. Let $\phi_t$ be the curve of local diffeomorphisms tangent to $X_t$. We have $\phi_t|_L=id|_L$.  So $\forall t\in [0,1]$ $\phi_t|_L$ is defined. But if $D(\phi)\subset [0,1]\times U$ is the domain of $\phi$, then $[0,1]\times L\subset D(\phi)$ so, by the openness of $D(\phi)$, we may suppose (again shrinking the domain $U$ if necessary), that $\phi_t$ is defined in $U$ for all $t\in[0,1]$ . Now we compute :
\begin{align*}
\frac{\partial}{\partial t} \left(\phi_t^*\omega_t\right) &=
\phi_t^*\left(\frac{\partial}{\partial t}\omega_t +\mathcal{L}_{X_t}\omega_t\right) \\
&=\phi_t^*(\omega'-d\mu)=0.
\end{align*}
Let $g:=\phi_1$. Then $g^*\widetilde{\omega}=\omega$, so if $\phi=g\circ f$ (where $f$ is defined above), we obtain $\phi^*(\omega^{\Lambda^k(T^*L)})=\omega$, and maintain $\phi|_L=id_L$, concluding the proof. 
\end{proof}


\appendix
\setcounter{secnumdepth}{0}
\section{Appendix: Two results in foliated differential topology}

In this appendix we give proofs for two "folkloristic" but subtle (and useful) extensions of well-known 
results in differential topology. Both are used in \cite{Martin1988} but ask for
a detailed proof. A brief sketch of a proof of the first result is given on the pages 88-89 in \cite{conl}. \\

We begin with the tubular neighborhood theorem, in the presence of a foliation:

\begin{thm}[Foliated tubular neighborhood theorem]
Let $M$ be a manifold, $W\subset TM$ an integrable distribution, and $N$ a submanifold complementary to $W$ in the sense that
 $W|_N\oplus TN = TM|_N$. Then there exist an open neighborhood $U$ of $N$ in $M$, and a diffeomorphism $\phi$ from $U$ 
 onto an open subset of $W|_N$ containing the zero section, such that $\phi|_N=id_N$, the differential of $\phi$ at any point of $N$ 
 is the identity, and $\phi$ maps for all $p$ in $N$ the leaf of the foliation defined by $W$ passing through it to the fiber 
 $\phi(U)\cap (W|_p)$ of $W\vert_N \to N$, intersected with $\phi(U)$. 
\end{thm}

\begin{proof}
Let $g$ be a fixed (auxiliary) Riemannian metric on the manifold $M$.\newline

\noindent Given $q\in M$, the leaf $\mathcal{W}_q$ of the foliation $\mathcal{W}$ defined by the distribution $W$ and containing $q$ is given 
as an injectively immersed submanifold $j_q:F_q\rightarrow M$ (with image $j_q(F_q)=\mathcal{W}_q$). The induced Riemannian 
metric $j_q^{*}(g)$ defines an exponential map $exp^{\mathcal{W}}$, notably one has $exp^{\mathcal{W}}_q:T_q(F_q)\rightarrow F_q$, 
defined on an open neighborhood of 0. Since $T_q(F_q)$ is canonically identified with $W_q=T_q(\mathcal{W}_q)$ via the differential of $j_q$, 
and $j_q$ is smooth, $exp^{\mathcal{W}}_q$ is a smooth map from an open neighborhood of 0 in $W_q$ to $M$, having values in 
$\mathcal{W}_q$. Restricting $q$ to be an element of $N$ we obtain a map $exp^{\mathcal{W},N}$ from a subset of $W|_N$ containing 
$N$ to $M$. 

Let us now show that $exp^{\mathcal{W},N}$ is, in fact, smooth on an open neighborhood of $N$ in $W|_N$. Fix $q$ in $ N$ and a 
coordinate chart $M\supset U\xrightarrow{\varphi} V_1\times V_2\subset \mathbb{R}^{m-d}\times\mathbb{R}^d$, such that the fibers 
of $\pi:V_1\times V_2\rightarrow V_1$ are the leaves of the foliation $\mathcal{W}$ ($d$ is the rank of this foliation). Furthermore, 
we can assume that $\varphi(q)=0$ and denote the elements of $\mathbb{R}^{m-d}$ resp. $\mathbb{R}^d$ by $x$ resp. $z$. 

We denote $\varphi(U\cap N)$ by $N$ and $T\varphi(W)$ by $W$ if no ambiguities are possible. By the assumption 
$TM|_N=TN\oplus W|_N$ we have $\forall q \in N \subset V_1\times V_2$ that $\mathbb{R}^m=T_q(V_1\times V_2)=
T_qN\oplus W_q=T_qN\oplus \mathbb{R}^d$ and thus the natural projection $\pi_q: T_qN\rightarrow \mathbb{R}^{m-d}$ 
is a linear isomorphism. Thus $\pi\vert_N: N\rightarrow V_1$ has everywhere maximal rank equal to the dimension of $V_1$. 
Shrinking $V_1$ and $V_2$ if necessary, we can assume that $\pi|_N:N\rightarrow V_1$ is a diffeomorphism whose inverse is 
described by $(id_{V_1},f):V_1\rightarrow V_1\times V_2$, where $f:V_1\rightarrow V_2$ is smooth and $N=\Gamma_f$, the 
graph of $f$. The map $\psi$ given by $\psi(x,z)=(x,z-f(x))=:(x,y)$ is a diffeomorphism of $V_1\times V_2$ to an open subset of 
$\mathbb{R}^m$. Restricting $\psi$ to an appropriate open neighborhood of 0, the image of $\psi$ is a product of open subsets 
of $\mathbb{R}^{m-d}$ and $\mathbb{R}^d$. Furthermore, $\psi(0)=0$, $\psi$ preserves the leaves of $\mathcal{W}$, and maps 
$N=\Gamma_f$ to $\{y=0\}$.

Post-composing $\varphi$ with $\psi$ yields a chart of $M$ near $q$ compatible with the foliation $\mathcal{W}$ and "adapted" to $N$. 
Obviously, we can construct a locally finite covering of $N$ by open subsets of $M$ that are domains of such charts, again 
denoted by $\varphi:U\rightarrow V_1\times V_2$ for simplicity.

In these coordinates  $exp^{\mathcal{W},N}_{(x,0)}$ is given as the time-one value of the solution of the following differential equation :

\begin{equation*}
\frac{d^2 y^k}{dt^2} + \sum_{i,j}\Gamma_{i,j}^k(x,y)\frac{d y^i}{dt} \frac{d y^j}{dt}=0   \,\,\,\,    \hbox{for}\, \,\, \,  1\leq k \leq d, 
\end{equation*}

\noindent subject to the initial condition that $x\in V_1$, $y(0)=0$ and $\frac{dy}{dt}(0)\in W_{(x,0)}$. Standard results on smooth ordinary 
differential equations depending smoothly on parameters imply that there exists an open subset $\widetilde{O}\subset W|_N$ 
containing $N$, where $exp^{\mathcal{W},N}$ is uniquely defined and smooth. 

Upon identifying, for $q\in N$, $T_q\widetilde{O}=T_qN\oplus W_q=T_qM$, we obtain that $D(exp^{\mathcal{W},N}_q)=id_{T_qM}$. 
By the below cited Proposition 3, it follows that there exists an open neighborhood $O$ of $N$ in $\widetilde{O}\subset W|_N$ 
such that $exp^{\mathcal{W},N}|_O$ is a diffeomorphism onto its image $U$, an open neighborhood of $N$ in $M$. Calling its inverse $\phi$, 
this latter map fulfills the conditions stated in Theorem 2.

\end{proof}	

\noindent The last argument relies on a standard result in differential topology (cf., e.g., Proposition 7.3 in \cite{golubguil}):

\begin{prop}
Let $Y$ and $Y'$ be two manifolds, and $X\subset Y$, $X'\subset Y'$ two regular submanifolds. Let $f:Y\rightarrow Y'$ be a smooth map satisfying :
\begin{itemize}
\item $f|_X:X\rightarrow X'$ is a diffeomorphism
\item $T_xf:T_xY\rightarrow T_{f(x)}Y'$ is an isomorphism for all $x\in X$
\end{itemize}
Then there exists an open neighborhood $V$ of $X$ in $Y$ such that $f(V)$ is open in $Y'$, and $f|_V$ is a diffeomorphism.
\end{prop}

Now we show the relative Poincar\'e lemma, again in the presence of a foliation:

\begin{thm} [Foliated relative Poincar\'e lemma]
Let $M$ be a smooth manifold and $N{\subset} M$ a submanifold. Let $\omega$ be a closed $(k{+}1)$-form on $M$ which vanishes when 
pulled back to $N$. Then there exists a neighborhood $U$ of $N$ in $M$, and a $k$-form $\mu$ defined on $U$, such that 
$d\mu=\omega|_U$ and $\mu|_N =0$. Moreover, if there exists an integrable distribution $W\subset TM$ complementary to $N$, and 
such that $\iota_{u\wedge v}\omega=0$ whenever $x$ is in $M$ and $u$ and $v$ are in $W_x \subset T_xM$, we may choose 
$\mu$ such that $\iota_{X}\mu=0$, for all vector fields $X$ taking value in $W$ and defined on an open subset of $U$.
\end{thm}

\begin{proof}
By the (standard) tubular neighborhood theorem, there exist $U$ and $V$ neighborhoods of $N$ in $M$ respectively $E$ (where $E\rightarrow N$ 
can be chosen to be any vector bundle such that $E\oplus TN=TM|_N$), and a diffeomorphism $\phi: U\rightarrow V$ fixing 
$N$ pointwise. Thus in what follows, 
we can and will assume to be in a open neighborhood $U$ of $N$ in $M$, which is also a vector bundle $\pi: E=U\to N$ over $N$. Let us consider 
the map:
\begin{equation*}
H:[0,1]\times U\rightarrow U\, ,  \, (t,x)\mapsto t\cdot x =tx.
\end{equation*}
If we denote $H_t(x):=H(t,x)$ then $H_0=\iota\circ \pi$ (where 
$\iota:N\rightarrow U=E$ is the inclusion of $N$ as the zero-section of $E$), and $H_1=id_E=id_U$. Let 
$Y_t(x):=\frac{d}{dt} H_t(x)$. The smooth map $Y_t$ is not a vector field since 
$H$ is not a flow, but the following formula still holds:
\begin{equation*}\tag{*}
\frac{d}{dt}(H_t^*\omega)=d(H_t^*\iota_{Y_t}\omega)+H_t^*\iota_{Y_t}d\omega,
\end{equation*} 
where, for $\alpha$ a $(k{+}1)$-form, $H_t^*\iota_{Y_t}\alpha$ is the following (well-defined!) $k$-form:
\begin{equation*}
(H_t^*\iota_{Y_t}\alpha)_x(v_1,...,v_k)=\alpha_{tx}(Y_t(x),T_xH_t(v_1),...,T_xH_t(v_k)),
\end{equation*}
for $x\in U$ and $v_j \in T_xU$. For a proof of $(*)$ see \cite{guilstern}. 

Since $\iota^*\omega=0$ and $\omega$ is closed we obtain:
\begin{align*}
\omega|_U&=H_1^*\omega-H_0^*\omega \\
&=\int_{[0,1]}\left(\frac{d}{dt}H_t^*\omega\right)dt \\
&=\int_{[0,1]}(d(H_t^*\iota_{Y_t}\omega))dt \\
&=d\int_{[0,1]}(H_t^*\iota_{Y_t}\omega)dt \\
&=d\mu,
\end{align*}
where we set $\mu:=\int_{[0,1]}(H_t^*\iota_{Y_t}\omega)dt$. Moreover $\mu|_N=0$ because $Y_t|_N=0$. 

To prove the last part of the theorem, we apply Theorem 2 in order to choose a foliated tubular 
neighborhood $U$ of $N$ with respect to $W$.
We can thus assume that the fibers of $U\rightarrow N$ are the fibers of $W|_N\rightarrow N$. Then for 
$x\in U$, $Y_t(x)\in W_{tx}$, implying for $X$ a vector field tangent to $W$:
\begin{align*}
(\iota_XH_t^*\iota_{Y_t}\omega)_x(v_1,...,v_{k{-}1})&=(H_t^*\iota_{Y_t}\omega)_x(X(x),v_1,...,v_{k{-}1}) \\
&=\omega_{tx}(Y_t(x),T_xH_t(X(x)),T_xH_t(v_1),...,T_xH_t(v_{k-1})) \\
&=0
\end{align*}
since $Y_t(x)$ and $T_xH_t(X(x))$ are both in $W_{tx}$. 
\end{proof}


\bibliographystyle{plain}

\bibliography{GS-TW-Lagrangian-smflds}

\end{document}